\documentclass[11pt,a4paper,reqno]{article}
\usepackage{amsmath}
\usepackage{amsthm}
\usepackage{enumerate}
\usepackage{amssymb}
\usepackage{mathrsfs}
\usepackage{xfrac}

\usepackage{verbatim}
\usepackage{url}
\usepackage[latin1]{inputenc}

\usepackage{caption}
\usepackage{graphicx,color}
\def\N{{\mathbb N}}
\def\Z{{\mathbb Z}}

\def\R{{\mathbb R}}

\def\Tb{{\mathbb T}}

\def\Ac{{\mathcal A}}

\def\Cc{{\mathcal C}}

\def\Hc{{\mathcal H}}

\def\Ps{{\mathscr P}}
\def\Ss{{\mathscr S}}

\def\be{\begin{equation}}
\def\ee{\end{equation}}
\def\bea{\begin{equation*}}
\def\eea{\end{equation*}}

\def\eps{\varepsilon}
\def\half{{\sfrac{1}{2}}}

\def\Pr{{\mathbb P}}
\DeclareMathOperator{\E}{{\mathbb E}}
\DeclareMathOperator{\Var}{Var}

\newtheorem{thm}{Theorem}
\newtheorem{lma}[thm]{Lemma}

\newtheorem{prop}[thm]{Proposition}

\newtheorem*{alg}{Algorithm}

\theoremstyle{remark}
\newtheorem*{remark}{Remark}
\newtheorem{preex}[thm]{Example}

\newtheorem*{keywords}{Keywords}
\newtheorem*{msc}{Mathematics Subject Classification 2010}

\theoremstyle{definition}
\newtheorem*{acknow}{Acknowledgements}

\begin{document}

\title{Partially observed Boolean sequences and noise sensitivity}
\date{}
\author{Daniel Ahlberg\\ \it Mathematical Sciences Research Institute\thanks{MSRI, 17 Gauss Way, Berkeley, CA 94720-5070, United States.}
}
\maketitle

\begin{abstract}
Let $\Hc$ denote a collection of subsets of $\{1,2,\ldots,n\}$, and assign independent random variables uniformly distributed over $[0,1]$ to the $n$ elements. Declare an element $p$-present if its corresponding value is at most $p$. In this paper, we quantify how much the observation of the $r$-present ($r>p$) set of elements affects the probability that the set of $p$-present elements is contained in $\Hc$. In the context of percolation, we find that this question is closely linked to the near-critical regime. As a consequence, we show that for every $r>1/2$, bond percolation on the subgraph of the square lattice given by the set of $r$-present edges is almost surely noise sensitive at criticality, thus generalizing a result due to Benjamini, Kalai and Schramm.

\begin{keywords}
Noise sensitivity, near-critical percolation, discrete Fourier analysis.
\end{keywords}

\begin{msc}
60C05, 60K35, 06E30.
\end{msc}
\end{abstract}

\section{Introduction}

Assign independent random weights, uniformly distributed on $[0,1]$, to the elements of $[n]:=\{1,2,\ldots,n\}$. Declare an element \emph{$p$-present} if its value is at most $p$, and denote the set of $p$-present elements by $\eta_p$. Given $p\in(0,1)$ and $\Hc\subset\{0,1\}^n$, assume that we are interested in the event $\{\eta_p\in\Hc\}$. The problem we address in this paper is whether we already at a larger `scale' $r\in(p,1)$ can tell whether $\{\eta_p\in\Hc\}$ will occur or not. The effect of an observation on a larger scale $r$ can be measured in terms of the variance of $\Pr(\eta_p\in\Hc|\,\eta_r)$. A small variance corresponds to a rather negligible effect of observing $\eta_r$, whereas a variance close to $\Pr(\eta_p\in\Hc)\Pr(\eta_p\not\in\Hc)$ bears witness of a rather decisive effect on the outcome of $\{\eta_p\in\Hc\}$.

Our first result gives an upper bound on $\Var\big(\Pr(\eta_p\in\Hc|\,\eta_r)\big)$, uniform in both $\Hc$ and $n$, and thus limiting the amount by which an observation at a larger scale $r$ may affect the outcome of $\Hc$ at scale $p$. We next investigate the circumstances under which
\be
\Var\big(\Pr(\eta_p\in\Hc|\,\eta_r)\big)\to0\quad\text{as }n\to\infty,
\ee
in which case we say that the observation at scale $r>p$ is asymptotically \emph{clueless} for determining $\Hc$ at scale $p$. Of course, here the property $\Hc$ will have to comprise a sequence $(\Hc_n)_{n\ge1}$ of properties $\Hc_n\subset\{0,1\}^n$. We will show that asymptotic cluelessness corresponds to the concept of noise sensitive, as introduced by Benjamini, Kalai and Schramm~\cite{benkalsch99} (a precise definition is given in Section~\ref{sec:clueless} below).

In their seminal work~\cite{benkalsch99}, the authors studied noise sensitivity in the context of bond percolation on the $\Z^2$ nearest-neighbour lattice. Loosely put, the model was found to be noise sensitive in the following sense: If we pick a critical configuration of bonds, meaning that each bond is independently present with probability $1/2$, and perturb the configuration slightly by independently flipping the state of each bond with low probability, then observing the perturbed configuration gives essentially no information regarding the existence of crossings over large scales in the original configuration.

The work of Benjamini, Kalai and Schramm was later refined by Schramm and Steif~\cite{schste10}, and Garban, Pete and Schramm~\cite{garpetsch10}. We will in this study see how the connection between cluelessness and noise sensitivity can be used to contribute further to the theory. First, we will show that bond percolation on a `typical' subgraph of the $\Z^2$ lattice of density $r>1/2$ is noise sensitive at criticality (see Theorem~\ref{thm:NSx}). Second, we think of $n$ as large but finite, and look for the point, as $r\to p$, the observation of $\eta_r$ contributes to the outcome of $\{\eta_p\in\Hc\}$. For bond percolation on the $\Z^2$ lattice, and for site percolation on the triangular grid, this point coincides with the boundary of the so-called near-critical regime of percolation (see Theorem~\ref{thm:nearcritical}).

Before presenting our results in greater detail, let us say a few words on inspiration and notation. The motivation for this study originates from questions that emerged in a recent study of noise sensitivity in continuum percolation~\cite{Abrogrimor}. The nature of these questions is discrete, and we will address them in the general setting of real-valued functions on the discrete cube $\{0,1\}^n$. A recurring theme will be that of discrete Fourier techniques. We follow the frequent convention that identifies subsets of $[n]$ with elements in $\{0,1\}^n$. We will let $\Pr_p$ denote product measure on $\{0,1\}^n$, which coordinate-wise gives mass $p$ to the value $1$, and let $\E_p$ and $\Var_p$ denote expectation and variance with respect to $\Pr_p$. It will in various instances be convenient to express the pair $(\eta_p,\eta_r)$ as $(\psi\cdot\xi,\psi)$, where the two configurations $\psi$ and $\xi$ in $\{0,1\}^n$ are chosen according to $\Pr_r$ and $\Pr_{p/r}$, respectively. Here and below, `$\cdot$' will 
refer to coordinate-wise multiplication. We will in general 
restrict the discussion to $r>p$, as the opposite case is obtained by a switch of zeros and ones.

\subsection{The combinatorial problem}

Our first result provides a sharp upper bound quantifying the average information given by the observation of $\eta_r$ regarding the outcome of $\{\eta_p\in\Hc\}$. The result is stated more generally in terms of real-valued functions on the discrete cube.

\begin{thm}\label{thm:var2step}
Let $f:\{0,1\}^n\to\R$, and $0<p<r<1$. Then
$$
\Var_r\big(\E_{\sfrac{p}{r}}\big[f(\psi\cdot\xi)\big|\,\psi\big]\big)\,\le\,\frac{p}{r}\,\frac{1-r}{1-p}\,\Var_p(f).
$$
\end{thm}

Theorem~\ref{thm:var2step} is sharp, as equality is attained by the dictator function (whose output is the value of the first bit). The upper bound shows that as soon as $r$ is strictly larger than $p$, then the information we retrieve from observing $\eta_r$ is insufficient for determining the outcome of $f(\eta_p)$ decisively, unless $f$ is degenerate. In addition, our proof of Theorem~\ref{thm:var2step} gives the following lower bound, also that sharp and attained by the parity function (encoding whether the number of $1$s is even or not):
$$
\Big(\frac{p}{r}\,\frac{1-r}{1-p}\Big)^n\Var_p(f)\,\le\,
\Var_r\big(\E_{\sfrac{p}{r}}\big[f(\psi\cdot\xi)\big|\,\psi\big]\big).
$$

Theorem~\ref{thm:var2step} improves a result of~\cite[Theorem~1.4]{Abrogrimor}, which under a rather restrictive relation between the parameters $n$, $p$ and $r$ gives an upper bound of order $(p/r)(\log r/p)^2$, as $p/r\to0$. That proof followed a combinatorial approach based on a second moment estimate via an inequality due to Bey~\cite{bey03}. The approach taken here is different, and shows how discrete Fourier analysis can be used to give a concise proof.

\subsection{Asymptotically clueless observations correspond to noise sensitivity}\label{sec:clueless}

We next discuss the possibility that an observation at scale $r>p$ is asymptotically clueless for the outcome $\Hc$ at scale $p$. The two states $\eta_p$ and $\eta_r$ will of course be highly correlated point-wise (at least for $r$ close to $p$), but $\eta_r$ may still carry very little information regarding $\{\eta_p\in\Hc\}$. This turns out to be the case when the outcome is noise sensitive.

The concept of noise sensitivity was introduced in the context of Boolean functions by Benjamini, Kalai and Schramm~\cite{benkalsch99}. Let $(f_n)_{n\ge1}$ be a sequence of functions such that $f_n:\{0,1\}^n\to[0,1]$. Given $\eps\in(0,1)$, let $\omega\in\{0,1\}^n$ be chosen according to $\Pr_p$, and let $\omega^\eps\in\{0,1\}^n$ denote the element obtained by re-sampling each coordinate independently with probability $\eps$, again according to $\Pr_p$. Then, the sequence $(f_n)_{n\ge1}$ is said to be \emph{noise sensitive} at intensity $p$ (NS$_p$ for short) if, for every $\eps>0$,
$$
\E_p\big[f_n(\omega)f_n(\omega^\eps)\big]-\E_p\big[f_n(\omega)\big]^2\to0\quad\text{as }n\to\infty.
$$

The observation that allows us to relate cluelessness to noise sensitivity is that for any $\eps\in(0,1)$ and real-valued function $f$ on $\{0,1\}^n$ we have: For $r=p/(1-\eps(1-p))$,
\be\label{eq:varNSequiv}
\E_p\big[f(\omega)f(\omega^\eps)\big]-\E_p\big[f(\omega)\big]^2\,=\,\Var_r\big(\E_{\sfrac{p}{r}}\big[f(\psi\cdot\xi)\big|\,\psi\big]\big).
\ee
As a corollary we obtain the following characterization of asymptotically clueless observations.

\begin{prop}\label{prop:varNSequiv}
Let $(f_n)_{n\ge1}$ be any sequence of functions $f_n:\{0,1\}^n\to[0,1]$. Then
$$
(f_n)_{n\ge1}\text{ is NS$_p$}\quad\Leftrightarrow\quad\lim_{n\to\infty}\Var_r\big(\E_{\sfrac{p}{r}}\big[f_n(\psi\cdot\xi)\big|\,\psi\big]\big)=0\text{ for every }r\in(p,1).
$$
\end{prop}

In addition, we will below, in Proposition~\ref{prop:var2step}, show that if the observation of $\eta_r$ is asymptotically clueless for the outcome of $f_n(\eta_p)$ for some $r\in(p,1)$, then it is for every $r\in(p,1)$.

\subsection{Noise sensitivity of percolation on random subgraphs of the square lattice}

The motivation behind the pioneering work of Benjamini, Kalai and Schramm was to study sensitivity to perturbation and dynamics in the context of bond (or site) percolation. Central to percolation theory is the study of long-range connections. A natural criteria for noise sensitivity of a percolation model can therefore be expressed in terms of existence of box-crossings.
We will here consider bond percolation on $\Z^2$, although analogous results hold for site percolation on the triangular lattice $\Tb$, both of which have critical parameter $p_c=1/2$. For a more extensive introduction to noise sensitivity and percolation we refer the reader to~\cite{garste12}.

Let $E$ denote the set of edges of the lattice $\Z^2$. Let $\Lambda_{n,m}$ denote the rectangle $[0,n]\times[0,m]$, and let $\Cc(\Lambda_{n,m})\subset\{0,1\}^E$ denote the event that there exists a horizontal crossing of present edges in $\Lambda_{n,m}$. Benjamini, Kalai and Schramm~\cite{benkalsch99} proved that the sequence of functions encoding the events $\Cc(\Lambda_{n+1,n})$ is noise sensitive at criticality, i.e.\ NS$_\half$.

Although Benjamini, Kalai and Schramm stated their result for crossings of rectangles of dimension $(n+1)\times n$, it is straightforward to verify that their proof is valid for any sequence of rectangles of bounded aspect ratios. More precisely, let $(\Lambda_{n,m(n)})_{n\ge1}$ be any sequence of rectangles of bounded aspect ratios, meaning that $m(n)/n$ is bounded away from both zero and infinity. Let $g_n:\{0,1\}^E\to\{0,1\}$ denote the indicator function of $\Cc(\Lambda_{n,m(n)})$. Then, the argument of~\cite{benkalsch99} shows that the sequence $(g_n)_{n\ge1}$ is NS$_\half$, and as an immediate consequence of Proposition~\ref{prop:varNSequiv}, we further see that for every $r\in(\frac{1}{2},1]$
\be\label{eq:clueless_crossing}
\Var\big(\E\big[g_n(\eta_\half)\big|\,\eta_r\big]\big)\to0\quad\text{as }n\to\infty.
\ee
Based on~\eqref{eq:clueless_crossing} we will show that bond percolation on a uniformly chosen subgraph of the square lattice (with prescribed edge density) is noise sensitive at criticality. A uniformly chosen subgraph of density $r\in(\frac{1}{2},1]$ is uniquely determined by an element $\psi\in\{0,1\}^E$ chosen according to $\Pr_r$. Since the critical probability for bond percolation on the square lattice is $1/2$, the critical probability for bond percolation on the subgraph determined by $\psi$ equals $1/2r$. We will prove the following extension of the result of~\cite[Theorem~1.2]{benkalsch99}.

\begin{thm}\label{thm:NSx}
For every $r\in(\frac{1}{2},1]$, the existence of bond percolation crossings on the subgraph of the square lattice determined by the set of $r$-present edges is almost surely noise sensitive at criticality. More formally,
$$
\lim_{n\to\infty}\E_{\sfrac{1}{2r}}\big[g_n(\psi\cdot\xi)g_n(\psi\cdot\xi^\eps)\big|\,\psi\big]-\E_{\sfrac{1}{2r}}\big[g_n(\psi\cdot\xi)\big|\,\psi\big]^2=0\quad\text{for $\Pr_r$-almost every }\psi\in\{0,1\}^E.
$$
\end{thm}

Note that $r=1$ retains the result of~\cite{benkalsch99}.

\subsection{Clueless observations outside the near-critical regime for percolation}

We know as of~\eqref{eq:clueless_crossing} that the information regarding $g_n(\eta_\half)$ we obtain from $\eta_r$ for any $r\neq1/2$ is insignificant for large $n$. We can turn the question around, and for large but finite $n$ ask at which point, as $r$ approaches $1/2$, we call tell whether $\eta_\half$ will contain a horizontal crossing of an $n\times n$-square. The answer to this question is as $r$ enters the so-called near-critical regime.

The \emph{near-critical regime} of percolation can be described in a variety of essentially equivalent ways, see e.g.~\cite{kesten87,nolin08}. We will here opt for a description in terms of the expected number of pivotal edges. Given $\omega\in\{0,1\}^n$, a bit $i\in[n]$ is said to be \emph{pivotal} for the function $f:\{0,1\}^n\to\{0,1\}$ if changing the state of $\omega$ at $i$ will change the outcome of $f(\omega)$.
For $(g_n)_{n\ge1}$ defined as above and $\omega\in\{0,1\}^E$, let $\Ps_n=\Ps_{n}(\omega)$ denote the set of pivotal edges for $g_n$, and $|\Ps_n|$ its size.

The near-critical regime can somewhat loosely be described as the regime around $p_c=1/2$ such that if $p$ changes from $p_c$ to some other value $r$ in this regime, then we expect few pivotal edges to change their states. As we approach criticality (either from above or from below), the expected number of edges pivotal at criticality that will change state from $\eta_r$ to $\eta_\half$ equals $|r-\frac{1}{2}|\cdot\E_\half|\Ps_n|$. In case $|r-\frac{1}{2}|\cdot\E_\half|\Ps_n|\ll1$ it is unlikely that any pivotal edge will change at all, and we can expect that $\eta_r$ will describe quite well the outcome of $g_n(\eta_\half)$. On the other hand, if $|r-\frac{1}{2}|\cdot\E_\half|\Ps_n|\gg1$, then we expect many pivotal edges to change there value, and that the left and right sides of $\Lambda_{n,m(n)}$ should therefore be well-connected in $\eta_r$. In addition, this leaves the possibility that the observation of $\eta_r$ will give little information as to the outcome of $g_n(\eta_\half)$, since many changes are still to 
occur.

Verifying that this heuristic reasoning is correct is straightforward in the former case, whereas it in the latter case is a consequence of a much deeper argument due to Garban, Pete and Schramm~\cite{garpetsch10}. We will based on their work show how to derive the following result.

\begin{thm}\label{thm:nearcritical}
Let $(g_n)_{n\ge1}$ be defined as above, encoding the existence of horizontal crossings for a sequence of rectangles of bounded aspect ratios, and let $(r_n)_{n\ge1}$ take values in $[0,1]$. If $|r_n-\frac{1}{2}|\cdot\E_\half|\Ps_n|\to\infty$ as $n\to\infty$, then
\be\label{eq:nc1}
\lim_{n\to\infty}\Var\big(\E\big[g_n(\eta_\half)\big|\,\eta_{r_n}\big]\big)=0.
\ee
On the other hand, if $|r_n-\frac{1}{2}|\cdot\E_\half|\Ps_n|\to0$ as $n\to\infty$, then
\be\label{eq:nc2}
\lim_{n\to\infty}\Pr\big(g_n(\eta_\half)\neq g_n(\eta_{r_n})\big)=0.
\ee
\end{thm}

It is known that the expected number of pivotal edges is equivalent to $n^2\alpha_4(n)$ up to multiplicative constants, where $\alpha_4(n)$ denotes the probability that at criticality there are four arms, alternating between primal and dual, connecting $[-1,1]^2$ to the boundary of $[-n,n]^2$. That is, for some universal constant $c>0$
$$
c\,n^2\alpha_4(n)\;\le\;\E_\half|\Ps_n|\;\le\;\frac{1}{c}\,n^2\alpha_4(n).
$$
(For $n\times n$-squares this was obtained in~\cite[combine~(7.3) and~(2.10)]{garpetsch10}, and the argument extends straightforwardly to rectangles with bounded aspect ratios.) The precise asymptotics of $\alpha_4(n)$ is not known for bond percolation on $\Z^2$. The best know estimates give a lower bound of order $n^{-(2-\delta)}$ and an upper bound of order $n^{-(1+\delta)}$, for some $\delta>0$ (see e.g.~\cite[(2.6)]{garpetsch10}). This shows that $\E_\half|\Ps_n|$ grows at least polynomially in $n$. Critical site percolation on the triangular lattice $\Tb$ is more precisely understood, due to SLE theory and Smirnov's theorem. It is there known that the corresponding objects $\alpha_4^\Tb(n)$ and $\Ps_n^\Tb$ behave as $\alpha_4^\Tb(n)=n^{-5/4+o(1)}$ and consequently $\E_\half|\Ps_n^\Tb|=n^{3/4+o(1)}$, as $n\to\infty$, see~\cite{smiwer01}.

We may unravel what the above discussion says regarding Theorem~\ref{thm:nearcritical}. For bond percolation on $\Z^2$ it gives the existence of a constant $\delta>0$ such that~\eqref{eq:nc1} holds for $r_n=\frac{1}{2}+n^{-\delta}$, whereas~\eqref{eq:nc2} holds for $r_n=\frac{1}{2}+n^{-(1-\delta)}$. For site percolation on $\Tb$, and $r_n=\frac{1}{2}+n^{-\delta}$, the analogous statement of~\eqref{eq:nc1} holds for all $\delta<3/4$, whereas~\eqref{eq:nc2} holds for all $\delta>3/4$.

\begin{remark}
The work behind Theorem~\ref{thm:nearcritical} was almost entirely carried out in~\cite{garpetsch10}. However, the statement presented here does not seem to have been previously known. The first part of Theorem~\ref{thm:nearcritical}, which is the deeper statement, is indeed an easy consequence of~\eqref{eq:varNSequiv} and the result of~\cite[Corollary~1.2]{garpetsch10} stating that if $\eps_n\cdot\E_\half|\Ps_n|\to\infty$, then
\be\label{eq:qNS}
\E_\half\big[g_n(\omega)g_n(\omega^{\eps_n})\big]-\E_\half\big[g_n(\omega)\big]^2\to0\quad\text{as }n\to\infty.
\ee
Nevertheless, we will in Section~\ref{sec:nearcritical} opt for a slightly more detailed proof, based on~\cite[Theorem~1.1]{garpetsch10}, in order to emphasize the role of discrete Fourier analysis.
\end{remark}


\section{Fourier-Walsh representation}

A tool that has turned out to be very useful in connection to the study of Boolean functions is discrete Fourier analysis. For $\omega\in\{0,1\}^n$, $p\in(0,1)$ and $i\in [n]$, define
\[
\chi_i^p(\omega):=\left\{
\begin{array}{cc}
-\sqrt{\frac{1-p}{p}} & \textrm{if } \omega_i=1, \\
\sqrt{\frac{p}{1-p}} & \textrm{otherwise. }
\end{array}
\right.
\]
Furthermore, for $S\subset [n],$ let $\chi_S^p(\omega):=\prod_{ i \in S } \chi_i^p(\omega)$. (In particular, $\chi_\emptyset^p$ is the constant 1.) By independence of bits, we observe that for $i \neq j$
$$
\E_p\big[\chi_i^p(\omega)\chi_j^p(\omega)\big] \,= \, \E_p\big[\chi_i^p(\omega)\big]\E_p\big[\chi_j^p(\omega)\big] \, = \,0,\quad\text{and}\quad\E_p\big[\chi_i^p(\omega)^2\big]\,=\,1.
$$
In fact, it follows that the set $\{\chi_S^p\}_{S\subset [n]}$ forms an orthonormal basis for the space of real-valued functions $f:\{0,1\}^n\to \R$, where $\{0,1\}^n$ is endowed with $\Pr_p$. Functions in this space can therefore be expressed using \emph{Fourier-Walsh representation}:
\begin{equation}\label{eq:FW}
f(\omega)=\sum_{S\subset [n]} \hat{f}^p(S)\chi_S^p(\omega),
\end{equation}
where $\hat{f}^p(S) := \E_p[ f \chi_S^p]$ are referred to as \emph{Fourier coefficients}. From now on, we will not stress that $S\subset [n]$ in the notation. Since $\{\chi_S^p\}_{S\subset[n]}$ is an orthonormal basis, we obtain from \eqref{eq:FW} that
\begin{equation}\label{eq:FWvar}
\Var_p(f)=\sum_{S\neq\emptyset}\hat{f}^p(S)^2.
\end{equation}
Noise sensitivity can be characterized in a similar way. First, note that $\E_p\big[ \chi_S^p(\omega) \chi_{S'}^p(\omega^\eps) \big] = 0$ if $S \neq S'$, and that $\E_p\big[\chi_S^p(\omega)\chi_S^p(\omega^\eps)\big] = (1-\eps)^{|S|}$, since the expectation is zero whenever at least one of the coordinates $\{ \omega_i : i \in S \}$ is re-randomized, and one otherwise. From~\eqref{eq:FW} it follows that
\be\label{eq:NSF}
\E_p\big[f_n(\omega) f_n(\omega^{\eps}) \big] - \E_p\big[f_n(\omega)\big]^2 \; =\; \sum_{S\neq \emptyset}\hat{f}_n^p(S)^2(1-\eps)^{|S|}.
\ee
It is now easily seen that for any $p\in(0,1)$ and sequence $(f_n)_{n\ge1}$ of functions $f_n:\{0,1\}^n\to[0,1]$
\bea
(f_n)_{n\ge1}\text{ is NS$_p$}\quad\Leftrightarrow\quad\lim_{n\to\infty} \sum_{0 < |S| \le k}\hat{f}_n^p(S)^2 = 0\text{ for each }k\in\N,
\eea
since the sum of Fourier coefficients squared in this case is uniformly bounded as of~\eqref{eq:FWvar}.

\section{The universal variance bound: Proof of Theorem~\ref{thm:var2step}}

Let $0<p<r<1$ and $f:\{0,1\}^n\to\R$ be given. The aim of this section is to relate $f$ and $\E_{\sfrac{p}{r}}\big[f(\psi\cdot\xi)\big|\,\psi\big]$ with respect to the two `scales' governed by $\Pr_p$ and $\Pr_r$. Define $h_f:\{0,1\}^n\to\R$,
$$
h_f(\psi):=\E_{\sfrac{p}{r}}\big[f(\psi\cdot\xi)\big|\,\psi\big].
$$
The key to prove Theorem~\ref{thm:var2step} will be to relate $f$ and $h_f$ in terms of their Fourier coefficients.

\begin{prop}\label{prop:hfq}
$\quad\displaystyle{\widehat{h_f}^r(S)\,=\,\left(\frac{p}{r}\,\frac{1-r}{1-p}\right)^{|S|/2}\hat{f}^p(S).}$
\end{prop}

\begin{proof}
First, write $f=\sum_S\hat{f}^p(S)\,\chi_S^p$, and observe that
\bea
\begin{aligned}
\widehat{h_{f}}^{r}(S')\;&=\;\E_{r}\big[h_{f}(\psi)\cdot\chi^{r}_{S'}(\psi)\big]\;=\;\E_{r}\big[\E_{\sfrac{p}{r}}\big[f(\psi\cdot\xi)\big|\,\psi\big]\chi^{r}_{S'}(\psi)\big]\\
&=\;\sum_S\,\hat{f}^p(S)\E_{r}\big[\E_{\sfrac{p}{r}}\big[\chi^p_S(\psi\cdot\xi)\big|\,\psi\big]\chi^{r}_{S'}(\psi)\big].
\end{aligned}
\eea
Thus, it suffices to determine the Fourier coefficients of $\E_{\sfrac{p}{r}}\big[\chi^p_S(\psi\cdot\xi)\big|\,\psi\big]$ for each $S$. Due to the independence between bits, it will suffice to consider the case of $S=S'$ and $|S|=1$. First,
\begin{equation*}
\E_{\sfrac{p}{r}}\big[\chi^p_i(\psi\cdot\xi)\big|\,\psi\big]\,=\,\left\{
\begin{aligned}
&\sqrt{\frac{p}{1-p}} & \text{if }\psi_i=0,\\
&\sqrt{\frac{p}{1-p}}\left(1-\frac{p}{r}\right)-\sqrt{\frac{1-p}{p}}\,\frac{p}{r}\,=\,\sqrt{\frac{p}{1-p}}\left(1-\frac{1}{r}\right) & \text{if }\psi_i=1.
\end{aligned}
\right.
\end{equation*}
Consequently,
\begin{equation*}
\begin{aligned}
&\E_{r}\big[\E_{\sfrac{p}{r}}\big[\chi^p_i(\psi\cdot\xi)\big|\,\psi\big]\chi^{r}_i(\psi)\big]\\
&\qquad=\;\sqrt{\frac{p}{1-p}}\,\sqrt{\frac{r}{1-r}}\,\Pr_{r}(\psi_i=0)-\sqrt{\frac{p}{1-p}}\left(1-\frac{1}{r}\right)\sqrt{\frac{1-r}{r}}\,\Pr_{r}(\psi_i=1)\\
&\qquad=\;\sqrt{r(1-r)}\,\sqrt{\frac{p}{1-p}}\left(1-\left(1-\frac{1}{r}\right)\right)\;=\;\sqrt{\frac{p}{1-p}}\,\sqrt{\frac{1-r}{r}}.
\end{aligned}
\end{equation*}
Recall that $\chi_S^p(\omega)=\prod_{ i \in S } \chi_i^p(\omega)$, and that each $\chi_i^p(\omega)$ only depends on $\omega_i$. Thus, by independence of bits, it follows that
\begin{equation*}
\E_{r}\big[\E_{\sfrac{p}{r}}\big[\chi^p_S(\psi\cdot\xi)\big|\,\psi\big]\chi^{r}_{S'}(\psi)\big]=\left\{
\begin{aligned}
&\sqrt{\frac{p}{r}\,\frac{1-r}{1-p}}^{|S|} &\text{if }S=S',\\
&0 & \text{if }S\neq S',
\end{aligned}
\right.
\end{equation*}
and the proof is complete.
\end{proof}

\begin{remark}
The relation between $h_ f$ and $f$ was for the special case $r=1/2$ explored in~\cite{Abrogrimor}, in order to obtain a reduction from biased product measure to the uniform case.
\end{remark}

\begin{proof}[\bf Proof of Theorem~\ref{thm:var2step}]
The statement will at this point easily follow from the comparison of Fourier coefficients in Proposition~\ref{prop:hfq} and the variance formula~\eqref{eq:FWvar}. Together they give
\begin{equation}\label{eq:var2step}
\Var_{r}\big(\E_{\sfrac{p}{r}}\big[f(\psi\cdot\xi)\big|\,\psi\big]\big)\;=\;\sum_{S\neq\emptyset}\widehat{h_{f}}^{r}(S)^2\;=\;\sum_{S\neq\emptyset}\left(\frac{p}{r}\,\frac{1-r}{1-p}\right)^{|S|}\hat{f}^p(S)^2.
\end{equation}
By assumption $p<r$, so $p(1-r)<r(1-p)$, which together with~\eqref{eq:var2step} and~\eqref{eq:FWvar} gives that
\bea
\Var_{r}\big(\E_{\sfrac{p}{r}}\big[f(\psi\cdot\xi)\big|\,\psi\big]\big)\;\le\;\frac{p}{r}\,\frac{1-r}{1-p}\,\sum_{S\neq\emptyset}\hat{f}^p(S)^2\;=\;\frac{p}{r}\,\frac{1-r}{1-p}\,\Var_p(f).
\eea
The lower bound follows in a similar fashion.
\end{proof}

Similarly, based on Proposition~\ref{prop:hfq}, we also find the following.

\begin{prop}\label{prop:var2step}
Let $(f_n)_{n\ge1}$ be a sequence of functions $f_n:\{0,1\}^n\to[0,1]$. If
$$
\lim_{n\to\infty}\Var_{r}\big(\E_{\sfrac{p}{r}}\big[f_n(\psi\cdot\xi)\big|\,\psi\big]\big)=0
$$
for some $r\in(p,1)$, then it does for every $r\in(p,1)$.
\end{prop}

\begin{proof}
For $0<p<r<1$, it is easily deduced from~\eqref{eq:var2step} that
$$
\lim_{n\to\infty}\Var_{r}\big(\E_{\sfrac{p}{r}}\big[f_n(\psi\cdot\xi)\big|\,\psi\big]\big)=0\quad\Leftrightarrow\quad\lim_{n\to\infty}\sum_{0<|S|\le k}\hat{f}_n^p(S)^2=0\text{ for each }k\in\N.
$$
Now, if the left-hand side holds for some $r\in(p,1)$, then the right-hand side also holds. But that implies that the left-hand side holds for every $r\in(p,1)$.
\end{proof}

\section{Cluelessness, noise sensitivity and near-critical percolation}\label{sec:nearcritical}

The goal of this section is to prove Proposition~\ref{prop:varNSequiv} and Theorem~\ref{thm:nearcritical}. We begin with the former.

\begin{proof}[\bf Proof of Proposition~\ref{prop:varNSequiv}]
Pick $\eps\in(0,1)$ and set $r=p/(1-\eps(1-p))$. Note that as $\eps$ increases from $0$ to $1$, $r$ simultaneously ranges over $(p,1)$. It therefore suffices to verify that~\eqref{eq:varNSequiv} holds.

Let $\omega$, $\psi$ and $\xi$ be specified according to $\Pr_p$, $\Pr_r$ and $\Pr_{\sfrac{p}{r}}$ respectively, and let $\xi^\prime$ be an independent copy of $\xi$. It is then easy to verify that the pair $(\psi\cdot\xi,\psi\cdot\xi^\prime)$ has the same distribution as $(\omega,\omega^\eps)$. Consequently,
$$
\E_p\big[f_n(\omega)f_n(\omega^\eps)\big]\;=\;\E_r\big[\E_{\sfrac{p}{r}}\big[f_n(\psi\cdot\xi)f_n(\psi\cdot\xi^\prime)\big|\,\psi\big]\big]\;=\;\E_r\big[\E_{\sfrac{p}{r}}\big[f_n(\psi\cdot\xi)\big|\,\psi\big]^2\big],
$$
where the latter equality follows from the independence between $\xi$ and $\xi^\prime$. Hence~\eqref{eq:varNSequiv} holds.
\end{proof}

We now turn to the proof of Theorem~\ref{thm:nearcritical}. In view of~\eqref{eq:NSF} and~\eqref{eq:var2step}, we see that the behaviour of Boolean functions can be understood by studying their Fourier spectrum. A precise description of the Fourier spectrum of critical percolation is indeed the major contribution of~\cite{garpetsch10}.

Given a Boolean function $f:\{0,1\}^n\to\{0,1\}$, the \emph{spectral sample} of $f$ refers to the random variable $\Ss_f$ on subsets of $[n]$ defined by $\Pr(\Ss_f=S):=\hat f^\half(S)^2/\E_\half[f^2]$. (We will only be interested in the case $p=1/2$.) What allows us to relate $\Ss_f$ and the pivotal set $\Ps_f$ are their marginal distributions, i.e.\ that $\Pr_\half(i\in\Ps_f)=\Pr(i\in\Ss_f)$, which implies that
\be\label{eq:P=S}
\E_\half|\Ps_f|\,=\,\E|\Ss_f|.
\ee
This connection was observed already in~\cite{kahkallin88}.

Recall that $(g_n)_{n\ge1}$ is defined as the sequence of functions encoding the existence of a horizontal crossing of some sequence of rectangles with bounded aspect ratios. We will for short write $\Ss_n$ for the spectral sample of $g_n$, as well as we write $\Ps_n$ for its pivotal set. It is the size of the spectral sample $\Ss_n$ we are after. An application of Markov's inequality shows that
$$
\Pr\big(|\Ss_n|>t\E|\Ss_n|\big)\,\le\,\frac{1}{t}.
$$
In fact, $|\Ss_n|$ is concentrated around its mean. This is the main results in~\cite[Theorems~1.1 and~7.4]{garpetsch10}, i.e.\ that
\be\label{eq:spectraltail}
\sup_{n\ge1}\Pr\big(0<|\Ss_n|<t\E|\Ss_n|\big)\to0\quad\text{as }t\to0.
\ee
Our proof of Theorem~\ref{thm:nearcritical} will be based on~\eqref{eq:spectraltail}.

\begin{proof}[\bf Proof of Theorem~\ref{thm:nearcritical}]
By symmetry, it will suffice to consider the case $r>1/2$. First, assume that $|r_n-\frac{1}{2}|\E_\half|\Ps_n|=|r_n-\frac{1}{2}|\E|\Ss_n|\to\infty$ as $n\to\infty$. By~\eqref{eq:spectraltail} we may for every $\gamma>0$ pick $t_0(\gamma)>0$ such that
$$
\sup_{n\ge1}\Pr\big(0<|\Ss_n|<t\E|\Ss_n|\big)<\gamma\quad\text{for all }t<t_0(\gamma).
$$
Fix $t\in(0,t_0(\gamma))$. By~\eqref{eq:var2step} we therefore obtain that
$$
\Var\big(\E\big[g_n(\eta_\half)\big|\,\eta_{r_n}\big]\big)\;=\;\sum_{S\neq\emptyset}\left(\frac{1-r_n}{r_n}\right)^{|S|}\hat g_n(S)^2\;\le\;\gamma+\E[g_n^2]\left(\frac{1-r_n}{r_n}\right)^{t\E|\Ss_n|}.
$$
Note that $2\big(r_n-\frac{1}{2}\big)t^2\E|\Ss_n|>1$ for all large enough $n$. As $\E|\Ss_n|=\E_\half|\Ps_n|\to\infty$ as $n\to\infty$, we have
$$
\left(\frac{1-r_n}{r_n}\right)^{t\E|\Ss_n|}\;\le\;\left(1-\frac{1}{t^2\E|\Ss_n|}\right)^{t\E|\Ss_n|}\;\le\; e^{-1/t}+\gamma
$$
for large $n$. Since $\gamma>0$ and $t>0$ were arbitrary, the first part of the theorem follows.

The proof of the second part of the statement is more direct, as it suffices to estimate the probability of a pivotal change. Note that if $g_n(\eta_\half)\neq g_n(\eta_{r_n})$, then there must have been a flip of a pivotal edge in the interval $(\frac{1}{2},r_n]$. That is, there must be a an edge $e\in\Ps_n$ which is present in $\eta_{r_n}$, but not in $\eta_\half$. Via the union bound, we find that
$$
\Pr\big(g_n(\eta_\half)\neq g_n(\eta_{r_n})\big)\;\le\;\sum_{e\in E}\Pr\big(e\in\Ps_n\text{ and }\xi_e\in(\tfrac{1}{2},r_n]\big)\;=\;(r_n-\tfrac{1}{2})\E_\half|\Ps_n|,
$$
from which the statement follows.
\end{proof}

We remark, in addition, that as $|r_n-\frac{1}{2}|\E_\half|\Ps_n|\to0$ as $n\to\infty$, it follows that
$$
\lim_{n\to\infty}\E\big[\Var\big(g_n(\eta_\half)\big|\,\eta_{r_n}\big)\big]=0.
$$

\section{Percolation on random subgraphs: Proof of Theorem~\ref{thm:NSx}}

In order to prove Theorem~\ref{thm:NSx} we will employ the approach introduced and used in~\cite{benkalsch99} to prove that $(g_n)_{n\ge1}$ is NS$_\half$. Recall that $(g_n)_{n\ge1}$ is the sequence encoding the existence of horizontal crossings in a sequence of rectangles with bounded aspect ratios. For the proof of the case $r<1$ we will need to assume that the result is known for $r=1$, which is precisely the result of~\cite{benkalsch99}. The case $r<1$ will follow from an iteration of the argument used to prove the case $r=1$. A similar approach was in~\cite{Abrogrimor} used to prove that percolation on certain random geometric graphs in $\R^2$ is noise sensitive at criticality.

First, it is necessary to recall some notation. A function $f:\{0,1\}^n\to\R$ is said to be \emph{monotone} if for any $\omega,\omega'\in\{0,1\}^n$ such that $\omega_i\le\omega'_i$ for all $i\in[n]$, then $f(\omega)\le f(\omega')$. A (deterministic) \emph{algorithm} is a rule which, given the bits of $\omega$ already queried, tells you which bit to query next. An algorithm is said to \emph{determine} $f$ if it determines $f(\omega)$ for each $\omega\in\{0,1\}^n$. Finally, for $p\in(0,1)$, the \emph{revealment} of an algorithm $\Ac$ with respect to $K\subset[n]$ is defined as
$$
\delta_p(\Ac,K):=\max_{i\in K}\Pr_p(\Ac\text{ queries bit $i$}).
$$

The deterministic algorithm approach was first introduced in the case $p=1/2$ by Benjamini, Kalai and Schramm~\cite{benkalsch99}. It was later generalized to biased product measure ($p\neq1/2$) in~\cite[Theorem~2.5]{Abrogrimor}. The core of the approach can be synthesized as follows.

\begin{prop}\label{prop:detalg}
Let $p\in(0,1)$ and $\ell\in\N$ be fixed, and let $(f_n)_{n\ge1}$ be a sequence of monotone functions $f_n:\{0,1\}^n\to[0,1]$. For each $n\ge1$, assume that $\Ac_1,\Ac_2,\ldots,\Ac_\ell$ are algorithms determining $f_n$, and that $K_1\cup K_2\cup\ldots\cup K_\ell=[n]$. Then $(f_n)_{n\ge1}$ is NS$_p$ if for each $i=1,2,\ldots,\ell$
$$
\delta_p(\Ac_i,K_i)\big(\log n\big)^6\to0\quad\text{as }n\to\infty.
$$
\end{prop}

Turning to our particular setting, fix $\psi\in\{0,1\}^E$ and define $g_n^\psi:\{0,1\}^E\to\{0,1\}$ as $g_n^\psi(\xi)=g_n(\psi\cdot\xi)$. Theorem~\ref{thm:NSx} states that $(g_n^\psi)_{n\ge1}$ is NS$_{\sfrac{1}{2r}}$ for $\Pr_r$-almost every $\psi\in\{0,1\}^E$. Hence, the aim is to apply Proposition~\ref{prop:detalg} to $(g_n^\psi)_{n\ge1}$ for $\psi$ fixed. (A similar statement was proven in~\cite[Theorem~1.3]{Abrogrimor} in a continuum setting.) Suitable algorithms are defined as follows.

\begin{alg}
Let $\psi$ and $\xi$ be in $\{0,1\}^E$, and let $V_0=\{0\}\times[m(n)]$ denote the left side of the rectangle $\Lambda_{n,m(n)}$. For $k\geq1$, define the algorithm $\mathcal{A}_R$ and the set $V_k$ inductively as follows:
\begin{enumerate}
\item Query all edges contained in $\Lambda_{n,m(n)}$ which has one endpoint in $V_{k-1}$ and one outside.
\item Let $U_k$ denote the set of neighbouring points to $V_{k-1}$ that are endpoints to an edge $e$ queried in the previous step, and for which $\psi_e\cdot\xi_e=1$. Set $V_k=V_{k-1}\cup U_k$.
\item Repeat the above steps until $V_k=V_{k-1}$. At this point $g_n^\psi(\xi)=g_n(\psi\cdot\xi)$ is known.
\end{enumerate}
\end{alg}

Analogously, $\Ac_L$ is obtained by interchanging left and right. Let $K_L$ denote the subset of edges in $E$ contained in left half of the rectangle $\Lambda_{n,m(n)}=[0,n]\times[0,m(n)]$, and let $K_R$ denote the remaining edges in $\Lambda_{n,m(n)}$ -- those to the right. We stress that $K_L$ and $K_R$ are allowed to overlap, and together cover the restriction of $E$ to $\Lambda_{n,m(n)}$. It is easy to see that an edge $e$ in $K_R$ will be queried by $\Ac_R$ if and only if there is a path of $\psi\cdot\xi$-present edges in $\psi$ from $V_0$ to one of the endpoints of $e$. Since each endpoint of $e$ lies in the right half of $\Lambda_{n,m(n)}$, existence of such a path is contained in the event that there is a path of $\psi\cdot\xi$-present edges from an endpoint of $e$ reaching the boundary of a square of side length $n$, centred at $e$. This event is recognized as a `one-arm'-event. We should emphasize that to prove Theorem~\ref{thm:NSx}, the algorithm considers $\psi$ as known, and its revealment 
depends in turn on $\psi$. That is, $\delta_p(\Ac_R,e)=\Pr_p(\Ac_R\text{ queries $\xi$ at $e$}|\,\psi)$.

\begin{lma}\label{lma:revealment}
For every $r>1/2$ and $\alpha>0$, there exists $\beta>0$ such that for every $e\in K_R$
$$
\Pr_r\big(\delta_{\sfrac{1}{2r}}(\Ac_R,e)> n^{-\beta}\big)<n^{-\alpha}
$$
for all sufficiently large $n$. (And analogously for $\Ac_L$ and $K_L$.)
\end{lma}

Before proving Lemma~\ref{lma:revealment}, let us see how Theorem~\ref{thm:NSx} follows via Proposition~\ref{prop:detalg}.

\begin{proof}[\bf Proof of Theorem~\ref{thm:NSx}]
Fix $r>1/2$. Lemma~\ref{lma:revealment} states that there exists $\beta>0$ such that for every $e\in K_R$
$$
\Pr_r\big(\delta_{\sfrac{1}{2r}}(\Ac_R,e)>n^{-\beta}\big)<n^{-4},\quad\text{for sufficiently large }n.
$$
Since the aspect ratios of the sequence of $\Lambda_{n,m(n)}$-rectangles are bounded, the number of edges in each of $K_L$ and $K_R$ is $O(n^2)$. It follows via the union bound that, for some $C<\infty$,
$$
\Pr_r\big(\delta_{\sfrac{1}{2r}}(\Ac_R,K_R)>n^{-\beta}\big)<C\,n^{-2},\quad\text{for sufficiently large }n.
$$
The Borel-Cantelli Lemma gives that $\delta_{\sfrac{1}{2r}}(\Ac_R,K_R)>n^{-\beta}$ for at most finitely many $n$, almost surely. Analogously, $\delta_{\sfrac{1}{2r}}(\Ac_L,K_L)>n^{-\beta}$ for at most finitely many $n$, almost surely. Thus, almost surely,
$$
\big[\delta_{\sfrac{1}{2r}}(\Ac_R,K_R)+\delta_{\sfrac{1}{2r}}(\Ac_L,K_L)\big](\log n)^6\to0\quad\text{as }n\to\infty.
$$
Clearly, $g_n$ depends only on edges in $\Lambda_{n,m(n)}$, and the set of such edges equals $K_L\cup K_R$. Hence, Proposition~\ref{prop:detalg} gives that $(g_n^\psi)_{n\ge1}$ is NS$_{\sfrac{1}{2r}}$ for $\Pr_r$-almost every $\psi\in\{0,1\}^E$.
\end{proof}

It remains to prove Lemma~\ref{lma:revealment}. We will for the proof need a variant of~\eqref{eq:clueless_crossing}. Recall that $\Cc(\Lambda_{n,m})$ denotes the event of a path of present edges crossing the rectangle $[0,n]\times[0,m]$ horizontally. Moreover, let $\Cc^\ast(\Lambda_{n,m})$ denote the event of a present \emph{dual} path crossing $[0,n]\times[0,m]$ horizontally. Here, as usual, an edge in the dual is present if and only if the corresponding edge in the graph is absent.

\begin{lma}\label{lma:rsw}
There exists a constant $c>0$ such that for every $r>1/2$
$$
\Pr_r\Big(\Pr_{\sfrac{1}{2r}}\big(\psi\cdot\xi\in\Cc^\ast(\Lambda_{3n\times n})\big|\,\psi\big)\in(c,1-c)\Big)\to1\quad\text{as }n\to\infty.
$$
\end{lma}

\begin{proof}
By~\eqref{eq:clueless_crossing} and Chebychev's inequality we find that
\be\label{eq:rectangle}
\Pr_r\Big(\big|\Pr_{\sfrac{1}{2r}}\big(\psi\cdot\xi\in\Cc(\Lambda_{n\times3n})\big|\,\psi\big)-\Pr_\half\big(\omega\in\Cc(\Lambda_{n\times3n})\big)\big|>\eps\Big)\to0\quad\text{as }n\to\infty.
\ee
Next, recall that there is a constant $c>0$ such that $\Pr_\half\big(\omega\in\Cc(\Lambda_{n\times3n})\big)\in(c,1-c)$, uniformly in $n$, due to the Russo-Seymour-Welsh argument (see e.g.~\cite[Corollary~3.5]{bolrio06}). Via~\eqref{eq:rectangle} it follows that $\Pr_{\sfrac{1}{2r}}\big(\psi\cdot\xi\in\Cc(\Lambda_{n\times3n})\big|\,\psi\big)\in(c/2,1-c/2)$ with probability approaching $1$ as $n\to\infty$. We further note that for each edge configuration, there is either a present horizontal crossing, or a present vertical dual crossing of the $n\times3n$-rectangle (but not both). In addition, the presence of a vertical dual crossing of an $n\times3n$-rectangle equals $\Cc^\ast(\Lambda_{3n\times n})$ in law. Consequently
$$
\Pr_r\Big(\Pr_{\sfrac{1}{2r}}\big(\psi\cdot\xi\in\Cc^\ast(\Lambda_{3n\times n})\big|\,\psi\big)\in(c/2,1-c/2)\Big)\to1\quad\text{as }n\to\infty,
$$
as required.
\end{proof}

\begin{proof}[\bf Proof of Lemma~\ref{lma:revealment}]
Let $\mathcal{O}_n$ denote the event of a present circuit in the annulus $[-3n,3n]^2\setminus[-n,n]^2$, surrounding the origin. Moreover, let $\mathcal{O}_n^\ast$ denote the corresponding dual event. If there is a present path connecting the origin to the boundary of $[-n/2,n/2]^2$, then it is clear that $\mathcal{O}_\ell^\ast$ cannot occur for any $\ell\le n/6$. From the discussion preceding Lemma~\ref{lma:revealment} we conclude that
$$
\delta_{\sfrac{1}{2r}}(\Ac_R,e)\;\le\;2\,\Pr_{\sfrac{1}{2r}}\bigg(\psi\cdot\xi\in\Big(\bigcup_{\ell\le n/6} \mathcal{O}_\ell^\ast\Big)^c\bigg|\,\psi\bigg).
$$
Let $\ell_k:=3^k$, and $N:=(\log n-\log6)/\log3$. Thus, it will suffice to show that for some $\beta>0$
\be\label{eq:suff}
\Pr_r\Bigg(\Pr_{\sfrac{1}{2r}}\Big(\psi\cdot\xi\in\bigcap_{k\le m}(\mathcal{O}_{\ell_k}^\ast)^c\Big|\,\psi\Big)>n^{-\beta}\Bigg)<n^{-\alpha},\quad\text{for large enough }n.
\ee
Note that for this choice of $\ell_k$, the events $\mathcal{O}_{\ell_1},\mathcal{O}_{\ell_2},\ldots$ are mutually independent.

Given $\alpha>0$, pick $\gamma>0$ such that $2^N\gamma^{N/4}\le n^{-\alpha}$ for all $n$. In the usual way, it follows from Lemma~\ref{lma:rsw} and the FKG-inequality that $\Pr_r\big(\Pr_{\sfrac{1}{2r}}(\psi\cdot\xi\in \mathcal{O}_\ell^\ast|\,\psi)>c^4\big)>1-\gamma$ for all large $n$. Now, if $\Pr_{\sfrac{1}{2r}}(\psi\cdot\xi\in \mathcal{O}_{\ell_k}^\ast|\,\psi)>c^4$ holds for at least $N/2$ of the $k$'s, then 
$$
\Pr_{\sfrac{1}{2r}}\Big(\psi\cdot\xi\in\bigcap_{k\le N}(\mathcal{O}_{\ell_k}^\ast)^c\Big|\,\psi\Big)\;=\;\prod_{k\le N}\Pr_{\sfrac{1}{2r}}\big(\psi\cdot\xi\in (\mathcal{O}_{\ell_k}^\ast)^c\big|\,\psi\big)\;\le\;(1-c^4)^{N/2}\;\le\; n^{-\beta},
$$
for some $\beta>0$. On the contrary, if $\Pr_{\sfrac{1}{2r}}(\psi\cdot\xi\in \mathcal{O}_{\ell_k}^\ast|\,\psi)>c^4$ fails for at least $N/2$ of the $k$'s, then at least $N/4$ of these $k$'s are at least $N/4$. However, when $n$ is large we have
$$
\Pr_r\Big(\#\big\{k\ge N/4:\Pr_{\sfrac{1}{2r}}(\psi\cdot\xi\in \mathcal{O}_{\ell_k}^\ast|\,\psi)<c^4\big\}\ge N/4\Big)\;\le\; 2^N\gamma^{N/4}\;\le\; n^{-\alpha},
$$
by the choice of $\gamma$. This proves~\eqref{eq:suff}.
\end{proof}

\begin{acknow}
Research in part carried out during the scientific programme \emph{Random Spatial Processes} at MSRI, and in part at a visit to IMPA supported by CNPq.
The author would like to express his gratitude to these two institutions, as well as to Simon Griffiths, Rob Morris and Jeff Steif, for helpful discussions and comments on an earlier version of the manuscript.
\end{acknow}

\bibliographystyle{plain}
\bibliography{bib-percolation}

\newcommand{\noopsort}[1]{}\def\cprime{$'$}
\begin{thebibliography}{10}

\bibitem{Abrogrimor}
Daniel Ahlberg, Erik Broman, Simon Griffiths, and Robert Morris.
\newblock Noise sensitivity in continuum percolation.
\newblock To appear in \emph{Israel J.\ Math.}

\bibitem{benkalsch99}
Itai Benjamini, Gil Kalai, and Oded Schramm.
\newblock Noise sensitivity of {B}oolean functions and applications to
  percolation.
\newblock {\em Inst. Hautes \'Etudes Sci. Publ. Math.}, 90:5--43, 1999.

\bibitem{bey03}
Christian Bey.
\newblock An upper bound on the sum of squares of degrees in a hypergraph.
\newblock {\em Discrete Math.}, 269(1-3):259--263, 2003.

\bibitem{bolrio06}
B{\'e}la Bollob{\'a}s and Oliver Riordan.
\newblock {\em Percolation}.
\newblock Cambridge University Press, New York, 2006.

\bibitem{garpetsch10}
Christophe Garban, G{\'a}bor Pete, and Oded Schramm.
\newblock The {F}ourier spectrum of critical percolation.
\newblock {\em Acta Math.}, 205(1):19--104, 2010.

\bibitem{garste12}
Christophe Garban and Jeffrey~E. Steif.
\newblock Noise sensitivity and percolation.
\newblock In {\em Probability and statistical physics in two and more
  dimensions}, volume~15 of {\em Clay Math. Proc.}, pages 49--154. Amer. Math.
  Soc., Providence, RI, 2012.

\bibitem{kahkallin88}
Jeff Kahn, Gil Kalai, and Nathan Linial.
\newblock The influence of variables on {B}oolean functions.
\newblock In {\em 29th Annual Symposium on Foundations of Computer Science},
  pages 68--80, 1988.

\bibitem{kesten87}
Harry Kesten.
\newblock Scaling relations for {$2$}{D}-percolation.
\newblock {\em Comm. Math. Phys.}, 109(1):109--156, 1987.

\bibitem{nolin08}
Pierre Nolin.
\newblock Near-critical percolation in two dimensions.
\newblock {\em Electron. J. Probab.}, 13:no. 55, 1562--1623, 2008.

\bibitem{schste10}
Oded Schramm and Jeffrey~E. Steif.
\newblock Quantitative noise sensitivity and exceptional times for percolation.
\newblock {\em Ann. of Math. (2)}, 171(2):619--672, 2010.

\bibitem{smiwer01}
Stanislav Smirnov and Wendelin Werner.
\newblock Critical exponents for two-dimensional percolation.
\newblock {\em Math. Res. Lett.}, 8(5-6):729--744, 2001.

\end{thebibliography}

\end{document}